\documentclass[a4paper,14pt]{article}

\usepackage{mathtools, nicematrix}
\usepackage[left=20mm, top=20mm, right=20mm, bottom=20mm, nohead]{geometry}
\usepackage[T1,T2A]{fontenc}
\usepackage[utf8]{inputenc}
\usepackage[english]{babel}
\usepackage{amsmath}
\usepackage{amsfonts,amssymb}
\usepackage{amsthm}
\usepackage[electronic]{ifsym}
\usepackage[all]{xy}
\usepackage{authblk}
\usepackage{graphicx}%
\usepackage{multirow}%
\usepackage{amsmath,amssymb,amsfonts}%
\usepackage{amsthm}%
\usepackage{mathrsfs}%
\usepackage[title]{appendix}%
\usepackage{xcolor}%
\usepackage{textcomp}%
\usepackage{manyfoot}%
\usepackage{booktabs}%
\usepackage{algorithm}%
\usepackage{algorithmicx}%
\usepackage{algpseudocode}%
\usepackage{listings}%

\theoremstyle{definition}
\newtheorem{definition}{Definition}[section]

\theoremstyle{theorem}
\newtheorem{theorem}{Theorem}[section]

\theoremstyle{lemma}
\newtheorem{lemma}{Lemma}[section]

\theoremstyle{proposition}
\newtheorem{proposition}{Proposition}[section]

\theoremstyle{remark}
\newtheorem*{remark}{Remark}

\theoremstyle{corollary}
\newtheorem{corollary}{Corollary}

\theoremstyle{problem}
\newtheorem{problem}{Problem}

\theoremstyle{hypothesis}
\newtheorem{hypothesis}{Conjecture}

\begin{document}
\title{Fragments of \sf{IOpen}}

\author{Konstantin Kovalyov}

\affil{
Phystech School of Applied Mathematics and Computer Science, Moscow Institute of Physics and Technology, Moscow, Russia

kovalev.ka@phystech.edu
}

\maketitle

\begin{abstract}
    In this paper we consider some fragments of $\mathsf{IOpen}$ (Robinson arithmetic $\mathsf Q$ with induction for quantifier-free formulas) proposed by Harvey Friedman and answer some questions he asked about these theories. We prove that $\mathsf{I(lit)}$ is equivalent to $\mathsf{IOpen}$ and is not finitely axiomatizable over $\mathsf Q$, establish some inclusion relations between $\mathsf{I(=)}, \mathsf{I(\ne)}, \mathsf{I(\leqslant)}$ and $\mathsf{I} (\nleqslant)$. We also prove that the set of diophantine equations solvable in models of $\mathsf I (=)$ is (algorithmically) decidable. 
\end{abstract}

\section{Introduction}
 Recall that $\mathsf{IOpen}$ consists of Robinson arithmetic $\mathsf{Q}$ with the induction schema for all quantifier free formulas. We assume that $\leq$ is a symbol in the signature of $\mathsf{Q}$. In December 2021, Harvey Friedman posed some interesting questions about $\mathsf{IOpen}$ \cite{fom919}. To formulate Friedman's questions, let us introduce these fragments of $\mathsf{IOpen}$: $\mathsf{I(lit)}$ is $\mathsf Q$ with induction schema for all atomic formulas and negated atomic formulas, $\mathsf{I(=)}$ is $\mathsf Q$ with induction schema for all formulas of the form $t = s$, where $t$ and $s$ are arithmetic terms, in the similar way are defined theories $\mathsf{I(\ne)}, \mathsf{I(\leqslant)}, \mathsf{I(\nleqslant)}$.
His questions concern relations between the following fragments with weaker induction: $\mathsf{I(lit)}$, $\mathsf{I(=)}, \mathsf{I(\ne)}, \mathsf{I(\leqslant)}$ and $\mathsf{I(\nleqslant)}$. 

Friedman stated the following questions:

\begin{enumerate}
    \item Is $\mathsf{I(lit)}$ weaker then $\mathsf{IOpen}$?
    \item What are relationships between $\mathsf{I(=)}, \mathsf{I(\ne)}, \mathsf{I(\leqslant)}, \mathsf{I(\nleqslant)}$?
    \item Are there interesting theorems that are equivalent to $\mathsf{I(lit)}$ over $\mathsf{Q}$?
\end{enumerate}

It is also interesting to consider theories $\mathsf{I(=, \ne)}$ and $\mathsf{IOpen(=)}$ (induction for quantifier-free formulas, containing only atomic formulas of the form $s = t$) and ask a similar question about their equivalence. 

In addition to these questions, we can also try to figure out decidability of set of Diophantine equations, that have a solution in theory $\mathsf{T}$, where $\mathsf{T}$ is one of our theories. Formally, this set is $D(\mathsf{T}) = \{(s, t) | \exists \mathcal M \vDash \mathsf{T} + \exists \vec{x} (s(\vec x) = t(\vec x))\}$. There are some results in this area:
\begin{itemize}
    \item $D(\mathsf Q)$ is decidable (see \cite{jerabek2016})
    \item Decidability of $D(\mathsf{IOpen})$ is not proved yet, there are partial results (see \cite{wilkie1978}, \cite{van_den_dries1980}, \cite{otero1990})
\end{itemize}

In Section 2 we prove that $\mathsf{IOpen} \equiv \mathsf{I(lit)}$ and $\mathsf{I(lit)}$ is not finitely axiomatizable, which answers questions 1 and 3 in the negative. 

In Section 3 we establish the following facts about the relationships of our weak fragments:

\begin{itemize}
    \item $\mathsf I(=) \nvdash \mathsf I(\ne), \mathsf I (\leqslant), \mathsf I(\nleqslant)$,
    \item $\mathsf I(\ne) \nvdash \mathsf I(\leqslant), \mathsf I(\nleqslant)$,
    \item $\mathsf I(\leqslant) \nvdash \mathsf I(=), \mathsf I(\ne), \mathsf I(\nleqslant)$,
    \item $\mathsf I(\ne) + \forall x \forall y (x + y = y + x) \vdash \mathsf I(=)$.
\end{itemize}

We show that $D(\mathsf{I(=)})$ is decidable and that $\mathsf{I(=)}$ proves $Th_=(\mathbb N)$ (all true identities in $\mathbb N$), but $\mathsf Q + Th_=(\mathbb N) \nvdash \mathsf{I(=)}$.

In Section 4 we state the problems remaining open.

\section{Preliminaries}

\begin{definition} [Robinson arithmetic]
    \textit{Robinson arithmetic} $\mathsf{Q}$ consists of the following axioms in the arithmetical language $\mathcal{L}_{ar} = (0, S, +, \cdot, \leqslant)$:
    
    \begin{enumerate}
        \item[(Q1)] $Sx \ne 0$
        \item[(Q2)] $Sx = Sy \rightarrow x = y$
        \item[(Q3)] $x \ne 0 \rightarrow \exists y (x = Sy)$
        \item[(Q4)] $x + 0 = x$
        \item[(Q5)] $x + Sy = S(x + y)$
        \item[(Q6)] $x \cdot 0 = 0$
        \item[(Q7)] $x \cdot Sy = x \cdot y + x$
        \item[(Q8)] $x \leqslant y \leftrightarrow \exists r (r + x = y)$
    \end{enumerate}
\end{definition}

\begin{definition}
    $\mathsf{IOpen}$ consists of $\mathsf Q$ and the induction schema for all quantifier free formulas in the language $\mathcal L_{ar}$, $\mathsf{I(lit)}$ consists of $\mathsf Q$ and induction schema for all literals in the language $\mathcal L_{ar}$ (i.e. atomic formulas and their negations). In the similar way we can define $\mathsf{I(=)}$, $\mathsf{I(\ne)}, \mathsf{I(\leqslant)}$ and $\mathsf{I(\nleqslant)}$.
\end{definition}

\begin{proposition}[\cite{hajek_pudlak_2017}, Theorem 1.10]
    The following formulas are provable in $\mathsf{IOpen}$:
    \begin{enumerate}
        \item[(1)] $x + y = y + x$,
        \item[(2)] $x + (y + z) = (x + y) + z$,
        \item[(3)] $x \cdot y = y \cdot x$,
        \item[(4)] $x(y + z) = x y + x z$,
        \item[(5)] $x(y z) = (x y) z$,
        \item[(6)] $x + y = x + z \rightarrow y = z$,
        \item[(7)] $x \leqslant y \vee y \leqslant x$,
        \item[(8)] $x \leqslant y \wedge y \leqslant x \rightarrow x = y$,
        \item[(9)] $(x \leqslant y \wedge y \leqslant z) \rightarrow x \leqslant z$,
        \item[(10)] $x \leqslant y \leftrightarrow x + z \leqslant y + z$,
        \item[(11)] $(z \ne 0 \wedge x z = y z) \rightarrow x = y$,
        \item[(12)] $z \ne 0 \rightarrow (x \leqslant y \leftrightarrow x z \leqslant y z)$.
    \end{enumerate}
\end{proposition}

\begin{remark}
    (1)-(5) can be proven in $\mathsf{I(=)}$.
\end{remark}

All rings and semirings in this paper will be commutative, associative with identity. Usually, structures will be denoted by calligraphic letters (such as $\mathcal{M, F, R}, \dots$), and their domains will be denoted by $M, F, R, \dots$. 

\begin{definition}
    Let $\mathcal M$ be a ring (semiring), $\leqslant$ be a linear order on $\mathcal M$. Then $(\mathcal M, \leqslant)$ is called \textit{an ordered ring} if $\forall x, y, z \in M (x \leqslant y \leftrightarrow x + z \leqslant y + z)$ and $\forall x, y, z \in M, z > 0 (x \leqslant y \leftrightarrow x \cdot z \leqslant y \cdot z)$. An ordered ring (semiring) is called \textit{discretely ordered} if $1$ is the least positive element (or, equivalently, there is no elements between $0$ and $1$).
\end{definition}

\begin{definition}
    Let $\mathcal{M} \subseteq \mathcal{R}$ be two ordered rings (with the same orderings) and $\mathcal M$ be discretely ordered. Then $\mathcal M$ is \textit{an integer part} of $\mathcal R$ if $\forall r \in R \exists m \in M (m \leqslant r < m + 1)$. Such an $m$ is called the integer part of $r$. Notation: $\mathcal M \subseteq^{IP} \mathcal R$.
\end{definition}

\begin{remark}
    Since $\mathcal M$ is discretely ordered, for every $r \in R$ its integer part is uniquely defined.
\end{remark}

\begin{theorem}[\cite{shepherdson1964}]
    Let $\mathcal{M}$ be a discretely ordered ring, $\mathcal M^+$ be the non-negative part of $\mathcal M$. Then, $\mathcal{M}^+ \vDash \mathsf{IOpen}$ iff $\mathcal M \subseteq^{IP} R(\mathcal M)$, where $R(\mathcal M)$ is the real closure of the ordered fraction field of $\mathcal M$.
\end{theorem}

\section{$\mathsf{IOpen} \equiv \mathsf{I(lit)}$ and $\mathsf{I(lit)}$ is not finitely axiomatizable}

\begin{proposition}
    Statements (1)-(12) from Proposition 1.1 are provable in $\mathsf{I(lit)}$.
\end{proposition}

\begin{proof}
    Recall these formulas:
    
    \begin{enumerate}
        \item[(1)] $x + y = y + x$,
        \item[(2)] $x + (y + z) = (x + y) + z$,
        \item[(3)] $x \cdot y = y \cdot x$,
        \item[(4)] $x(y + z) = x y + x z$,
        \item[(5)] $x(y z) = (x y) z$,
        \item[(6)] $x + y = x + z \rightarrow y = z$,
        \item[(7)] $x \leqslant y \vee y \leqslant x$,
        \item[(8)] $x \leqslant y \wedge y \leqslant x \rightarrow x = y$,
        \item[(9)] $(x \leqslant y \wedge y \leqslant z) \rightarrow x \leqslant z$,
        \item[(10)] $x \leqslant y \leftrightarrow x + y \leqslant x + z$,
        \item[(11)] $(z \ne 0 \wedge x z = y z) \rightarrow x = y$,
        \item[(12)] $z \ne 0 \rightarrow (x \leqslant y \leftrightarrow x z \leqslant y z)$.
    \end{enumerate}

    As noted in Remark after Proposition 1.1, (1)-(5) are provable in $\mathsf{I(=)}$. We outline the proofs of (6)-(12).
    
    \begin{enumerate}
        
    \item[(6)] $x + y = x + z \rightarrow y = z$.
     
    Suppose $y \neq z$. We prove by induction on $x$ the statement $x + y \neq x + z$.
    
    If $x = 0$, $0 + y = y \neq z = 0 + z$ (here we used commutativity of addition and Q4).
    
    Let $x + y \neq x + z$. Then, $S x + y = S(x + y) \neq S(x + z) = S x + z$ (here we used commutativity, Q2 and Q5).
    
    \item[(7)] $x \leqslant y \vee y \leqslant x$.
    
    Suppose there exist $x, y$ such that $x \nleqslant y$ and $y \nleqslant x$. We prove $x \nleqslant y + z$ by induction on $z$.
    
    If $z = 0$, then $x \nleqslant y = y + 0$.
    
    Let $x \nleqslant y + z$. Suppose, $x \leqslant y + S z$. Then, there exists an $r$ such that $r + x = y + S z$. If $r = 0$, $x = y + S z$, then, $y \leqslant x$, and we have a contradiction. Let $r = S r'$. $S(r' + x) = S r' + x = y + Sz = S(y + z) \Rightarrow r' + x = y + z$. So, $x \leqslant y + z$, a contradiction.
    
    Now, let $z$ be $x$. Then $x \nleqslant y + x$, a contradiction.
    
    \item[(8)-(9)] Could be easily done, using commutativity and associativity of addition and axioms of $\mathsf Q$.
    
    \item[(10)] $x \leqslant y \leftrightarrow x + z \leqslant y + z$
    
    If $x \leqslant y$, then $r + x = y$ for some $r$, so $r + (x + z) = y + z$ and $x + z \leqslant y + z$.
    
    Suppose, $x + z \leqslant y + z$, but $x \nleqslant y$. By (7), $y \leqslant x$. Since we've already proved the opposite implication, $y + z \leqslant x + z$. Then, by (8), $x + z = y + z$. Using (6), we obtain that $x = y$, so $x \leqslant y$.
    
    \item[(11)] $(z \ne 0 \wedge x z = y z) \rightarrow x = y$
    
    Suppose, $x \ne y$. By (7) we can assume, that, for example, $x \leqslant y$. Then, there is $r \ne 0$ such that $r + x = y$. Suppose, $x z = y z$, where $z \ne 0$. Then, $x z = (r + x) z$, by (6) and distributivity, $r z = 0$, which is impossible, since $z \ne 0$ and $r \ne 0$.
    
    \item[(12)] $z \ne 0 \rightarrow (x \leqslant y \leftrightarrow x z \leqslant y z)$
    
    Suppose, $x \leqslant y$, then $r + x = y$ for some $r$. Then $y z = r z + x z$, so $x z \leqslant y z$.
    
    Using (7), we can prove the opposite implication.
    \end{enumerate}
\end{proof}

So, every model of $\mathsf{I(lit)}$ is a discretely ordered semiring. Let $\mathcal M = (M, +, \cdot, \leqslant, 0, 1)$ be a model of $\mathsf{I(lit)}$. We can extend this semiring to a ring in the following way. Consider pairs $(m, n)$ of elements of our semiring and define the equivalence relation on them: $(m, n) \sim (m', n') \leftrightharpoons m + n' = m' + n$ ($(m, n)$ can be understood as $m - n$). It easy to see that it is an equivalence relation. So, let $\widetilde{M} = M^2 / \sim$ and $\widetilde {\mathcal M} = (\widetilde{M}, ...)$ with the operations defined in an obvious way. It will be a discretely ordered ring and hence an integral domain. Denote by $F(\mathcal M)$ the (ordered) quotient field of $\widetilde{\mathcal M}$, by $R(\mathcal M)$ -- the real closure of $F(\mathcal M)$.

\begin{lemma}
    Let $f \in \widetilde{\mathcal M}[X]$, $f(\frac{a}{q}) \leqslant 0$, $f(\frac{b}{q}) > 0$, $a, b, q \in M$, $a < b$. Then $\exists c \in M: f(\frac{c}{q}) \leqslant 0 \wedge f(\frac{c + 1}{q}) > 0$.
\end{lemma}

\begin{proof}
    Define $g \in \widetilde{\mathcal M}[X]$ in the following way: $g(X) := q^n f(\frac{X + a}{q})$, where $n = \deg f$. Then $g(0) \leqslant 0$, $g(b - a) > 0$ and $\mathcal M \vDash g(0) \leqslant 0 \wedge \exists c (g(c) > 0)$. Since $\mathcal M \vDash \mathsf{I(lit)}$, $\mathcal M \vDash \exists c (g(c) \leqslant 0 \wedge g(c + 1) > 0)$. Then $f(\frac{c + a}{q}) \leqslant 0$ and $f(\frac{c + a + 1}{q}) > 0$.
\end{proof}

\begin{lemma}
    Let $f \in \widetilde{\mathcal M}[X]$, $f(\frac{a}{q}) < 0$, $f(\frac{b}{q}) > 0$, $a, b, q \in M$, $a < b$ and $f$ has no roots in $F(\mathcal M)$. Then $\exists c \in M: a \leqslant c < b \wedge f(\frac{c}{q}) < 0 \wedge f(\frac{c + 1}{q}) > 0$.
\end{lemma}

\begin{proof}
    Fix $a, b, q \in M$. Let $N(f) = \{m \in M|f(\frac{m}{q}) f(\frac{m + 1}{q}) < 0\}$. Note that $|N(f)|$ is finite, since for every $m \in N(f)$ there exists a root of $f$ in $R(\mathcal M)$ between $\frac{m}{q}$ and $\frac{m + 1}{q}$. 
    
    Suppose, there is $f \in \widetilde{\mathcal M}[X]$ such that $f(\frac{a}{q}) < 0$, $f(\frac{b}{q}) > 0$, but there is no $c$ between $a$ and $b$ such that $f(\frac{c}{q}) < 0$ and $f(\frac{c + 1}{q}) > 0$. Choose such an $f$ with the minimal $|N(f)|$. By Lemma 2.1 there is a $c \in M$ such that $f(\frac{c}{q}) < 0$ and $f(\frac{c + 1}{q}) > 0$. By the choice of $f$, $c < a$ or $c > b$.
    
    If $c > b$, consider $\tilde f(X) := f(X)((2 c + 1) - 2 q X)$. Then $\tilde f(\frac{a}{q}) = f(\frac{a}{q})(\frac{2 c - 2 a + 1}{q}) < 0$, $\tilde f(\frac{b}{q}) = f(\frac{b}{q})(\frac{2 c - 2 b + 1}{q}) > 0$, there is no such $\tilde c$ between $a$ and $b$ such that $\tilde f(\frac{\tilde c}{q}) \tilde f(\frac{\tilde c + 1}{q}) < 0$ (on $[a, b]$ $\tilde f$ has the same sign as $f$) and $N(f) = N(\tilde f) \setminus \{c\}$. Hence, we have a contradiction with the choice of $f$. If $c < a$, we can consider in a similar way $\tilde f(X) := f(X)(2 q X - (2 c + 1))$.
\end{proof}

\begin{theorem}
    Let $\mathcal M \vDash \mathsf{I(lit)}$. Then $\widetilde{\mathcal M} \subseteq^{IP} R(\mathcal M)$.
\end{theorem}

\begin{proof}
    First we prove that $\widetilde{\mathcal M} \subseteq^{IP} F(\mathcal M)$. Consider $\frac{p}{q} \in F(\mathcal M), p > 0, q > 0$ (it is sufficient to prove the existence of the integer parts only for positive elements of $F(\mathcal M)$). $\mathcal M \vDash 0\cdot q \leqslant p \wedge (p + 1)q > p$. Then, by induction, we obtain $\mathcal M \vDash \exists m (m q \leqslant p \wedge (m + 1)q > p)$.

    Consider a positive $r \in R(\mathcal M) \setminus F(\mathcal M)$. Let $f \in \widetilde{\mathcal{M}}[X]$ be the minimal polynomial of $r$. Let introduce the following equivalence relation $\sim$ on $R(\mathcal M)$: $x \sim y \leftrightharpoons \nexists z \in F(\mathcal M): (x < z < y \vee y < z < x)$. Note that if $x \sim y$ and $q_1 < x < q_2$ for some $q_1, q_2 \in F(\mathcal M)$, then $q_1 < y < q_2$. It is not very hard to prove that elements of $F(\mathcal M)$ can be equivalent only to themselves. If $f$ has some root $r' \sim r, r' \ne r$, then $f'$ has a root $r''$ between $r'$ and $r$ by Rolle's theorem (and $r'' \sim r$). If $r$ is a multiple root of $f$, then $f'(r) = 0$. So, we can take a derivative of $f$ until $f^{(k)}$ has only one simple root $\tilde r \sim r$. Then we can find $q_1, q_2 \in F(\mathcal M)$, $q_1, q_2 > 0$ such that the only root of $f^{(k)}$ between $q_1$ and $q_2$ is $\tilde r$ (since $r$ is positive, so is $\tilde r$, hence $q_1, q_2$ can be chosen positive). Let $q_i = \frac{a_i}{q}$, $q, a_i \in M$. Since $\tilde r$ is simple, $f^{(k)}(q_1) f^{(k)}(q_2) < 0$. Also we can suppose that $f^{(k)}$ has no roots in $F(\mathcal M)$ (if not, we can divide $f^{(k)}$ by $(X - q)$ for the suitable $q \in F(\mathcal M)$ and then multiply by the suitable $m \in M$). So, we can apply Lemma 2.2 and obtain that there exists $b \in M$ such that $f^{(k)}(\frac{b}{q}) f^{(k)}(\frac{b + 1}{q}) \leqslant 0$, $a_1 \leqslant b < a_2$. This implies that there is a root between $\frac{b}{q}$ and $\frac{b + 1}{q}$. Since there is only one root $\tilde r$ on the segment $[q_1, q_2]$, $\frac{b}{q} \leqslant \tilde r \leqslant \frac{b + 1}{q}$. Given that $\widetilde{\mathcal M} \subseteq^{IP} F(\mathcal M)$, we obtain that $\tilde r$ (and hence $r$) has an integer part in $\mathcal M$.
\end{proof}

\begin{corollary}
    $\mathsf{I(lit)} \vdash \mathsf{IOpen}$.
\end{corollary}

\begin{proof}
    Apply Theorem 1.1 to Theorem 2.1.
\end{proof}

\begin{theorem}
    $\mathsf{I(lit)}$ is not finitely axiomatizable.
\end{theorem}

\begin{proof}
    Suppose that $\mathsf{I(lit)}$ is finitely axiomatizable, then $\mathsf{I(lit)} \equiv \mathsf Q + \Gamma$, where $\Gamma$ is finite set of instances of induction axiom schema for literals. Denote by $N$ the largest degree of polynomials from $\Gamma$ (all terms in $\mathsf{I(lit)}$ are equal to polynomials) and denote by $p_1, \dots, p_n$ all the prime numbers $\leqslant N$. 
    
    Consider the following structure $\mathcal{M}$: $M = \{a_m X^{\frac{m}{q}} + a_{m - 1} X^{\frac{m - 1}{q}} + \dots + a_1 X^{\frac{1}{q}} + a_0| m, q \in \mathbb N, q = p_1^{\alpha_1}\dots p_n^{\alpha_n}$ for some $\alpha_1, \dots, \alpha_n \in \mathbb N, a_m, \dots, a_1 \in \mathbb R_{alg}, a_0 \in \mathbb Z, a_m \geqslant 0\}$ with the operations defined in the usual way. Note that the corresponding ring $\widetilde{\mathcal M}$ is not contained as an integer part of the real closure of the fraction field of this ring. We denote this real closure by $\mathcal R$ (in our case $\mathcal R = \{a_m X^{\frac{m}{q}} + a_{m - 1} X^{\frac{m - 1}{q}} + \dots + a_1 X^{\frac{1}{q}} + a_0 + a_{-1}X^{-\frac{1}{q}} + \dots | a_i \in \mathbb R_{alg}\}$, because of well known fact that the real closure of $\mathbb Z[X]$ is $\{a_m X^{\frac{m}{q}} + a_{m - 1} X^{\frac{m - 1}{q}} + \dots + a_1 X^{\frac{1}{q}} + a_0 + a_{-1}X^{-\frac{1}{q}} + \dots | a_i \in \mathbb R_{alg}\}$ and the latter contains $\mathcal M$). So, it is sufficient to prove $\mathcal M \vDash \mathsf Q + \Gamma$ and to apply Theorem 1.1. 
    
    \begin{lemma}
        Let $f \in \widetilde{\mathcal M}[t]\setminus\{0\}$, $\deg f \leqslant N$ and $r = a_m X^{\frac{m}{q}} + a_{m - 1} X^{\frac{m - 1}{q}} + \dots + a_1 X^{\frac{1}{q}} + a_0 + a_{-1} X^{-\frac{1}{q}} + \dots \in R$ be a root of $f$, $m > 0$. Then $\frac{m}{q} = \frac{m'}{q'}$, where $q' = p_1^{\alpha_1} \dots p_n^{\alpha_n}$ (i.e. $a_m X^{\frac{m}{q}} \in \widetilde{M}$). 
    \end{lemma}
    
    \begin{proof}[Proof of Lemma 2.3.]
        Let $f(t) = P_k t^k + \dots + P_0$, where $P_i \in \widetilde{M}$, so $P_k r^k + \dots + P_0 = 0$. All nonzero $P_i r^i$ are of the form 
        $$b_i X^{\frac{i\cdot m}{q} + \frac{k_i}{C}} + \dots,$$ 
        where $C = p_1^{\beta_1}\dots p_n^{\beta_n}$ is a common denominator of degrees in all $P_i$ (i.e. $f \in \mathbb R_{alg} [X^{\frac{1}{C}}][t]$), $b_i \in \mathbb R_{alg}\setminus\{0\}$.
        
        Consider the largest $\frac{i\cdot m}{q} + \frac{k_i}{C}$. Since $f(r) = 0$, there is $j \neq i$ such that $$\frac{j\cdot m}{q} + \frac{k_j}{C} = \frac{i\cdot m}{q} + \frac{k_i}{C}.$$
        So, $\frac{m}{q} = \frac{k_i - k_j}{C(j - i)}$. Let assume $j > i$, then, $m' := k_i - k_j$ and $q' := C(j - i)$. Since $j - i \leqslant N$, $q'$ is of the required form.
    \end{proof}
    
    \begin{lemma}
        Let $f \in \widetilde{\mathcal M}[t]$, $\deg f \leqslant N$ and $r = \sum\limits_{k = m}^{-\infty}a_k X^{\frac{k}{q}}\in R$, $f(r) = 0$. Then $r$ has an integer part in $\widetilde{\mathcal M}$.
    \end{lemma}
    
    \begin{proof}[Proof of Lemma 2.4.]
        Induction by $\max(m, 0)$. If $m \leqslant 0$, then $r \in (a_0 + 1, a_0 - 1)$ and $r$ has an integer part. If $m > 0$, $a_m X^{\frac{m}{q}} \in \widetilde{M}$ by Lemma 2.3. So, we can apply induction hypothesis to $f(t + a_m X^{\frac{m}{q}})$ and $r - a_m X^{\frac{m}{q}} = \sum\limits_{k = m - 1}^{-\infty}a_k X^{\frac{k}{q}}$. Denote by $s$ the integer part of $r - a_m X^{\frac{m}{q}}$, then $a_m X^{\frac{k}{m}} + s$ will be the integer part of $r$.
    \end{proof}
    
    \textit{Proof of Theorem 2.2.}
    Let $\varphi(x, \vec{y})$ be an atomic formula or the negation thereof such that $Ind_{\varphi} \in \Gamma$. Then, $\varphi$ is equivalent to one of the following: $f(x) = 0$, $f(x) \neq 0$, $f(x) \leqslant 0$, $f(x) < 0$, where $f \in \widetilde{\mathcal M}[t]$ (with the coefficients dependent on $\vec{y}$) and $\deg f \leqslant N$. Cases $f(x) = 0$ and $f(x) \neq 0$ are trivial (since if polynomial has an infinite number of roots, then it is a zero polynomial). Consider the case $f(x) \leqslant 0$, the case $f(x) < 0$ is very similar. Suppose, $\mathcal M \vDash (f(0) \leqslant 0) \wedge \exists c (f(c) > 0)$. Let $A = \{r \in R | f(r) > 0 \wedge r > 0\}$. Since $\mathcal R$ is real closed, $A$ is a finite union of disjoint intervals. Since $\mathcal M \vDash \exists c (f(c) > 0)$, $M \cap A \ne \varnothing$. Consider the leftmost interval $(a, b)$ of $A$ containing some element $c$ of $M$. Since $f(a) = 0$, we have $[a] \in M$, where $[a]$ is the integer part of $a$ (by Lemma 2.4). Since $[a] \leqslant a < [a] + 1$ and $\mathcal M$ is discretely ordered, $[a] + 1 \leqslant c$ and $[a] + 1 \in (a, b)$. So, $\mathcal M \vDash f([a]) \leqslant 0 \wedge f([a] + 1) > 0$.
\end{proof}

\section{Relations between $\mathsf{I(=)}$, $\mathsf{I(\ne)}$ and $\mathsf{I(\leqslant)}$}

Our aim in this section is to prove the following theorems:

\begin{theorem}
    There are the following relations between considered fragments:
    \begin{itemize}
        \item $\mathsf I(=) \nvdash \mathsf I(\ne), I\mathsf (\leqslant), \mathsf I(\nleqslant)$,
        \item $\mathsf I(\ne) \nvdash \mathsf I(\leqslant), \mathsf I(\nleqslant)$,
        \item $\mathsf I(\leqslant) \nvdash \mathsf I(=), \mathsf I(\ne), \mathsf I(\nleqslant)$,
    \end{itemize}
\end{theorem}

\begin{theorem}
    \begin{enumerate}
        \item[(i)] $D(\mathsf{I}(=))$ is decidable;
        \item[(ii)] $\mathsf I(=) \vdash Th_=(\mathbb N)$;
        \item[(iii)] $\mathsf Q + Th_=(\mathbb N) \nvdash \mathsf I(=)$.
    \end{enumerate}
\end{theorem}

\begin{theorem}
    $\mathsf{I}(\ne) + \forall x \forall y (x + y = y + x) \vdash \mathsf{I}(=)$.
\end{theorem}

\begin{proposition}
    (i) $\mathsf{I(=)} \nvdash Sx \ne x$ and $\mathsf{I(=)} \nvdash x + z = x + y \rightarrow z = y$;
    
    (ii) $\mathsf{I(=)} \nvdash \mathsf{I(\ne)}$.
\end{proposition}

\begin{proof}
    (i) Consider the $\mathcal L_{ar}$-structure $\mathcal M$ with the universe $M = \mathbb N \cup \{\omega\}$ and the operations defined in the following way: on natural numbers operations are defined in the standard way, $S\omega = \omega$, $x + \omega = \omega + x = \omega$, $0 \cdot \omega = \omega \cdot 0 = 0$, $x \ne 0 \rightarrow x \cdot \omega = \omega \cdot x = \omega$.
    
    It is easy to see that $\mathcal M \vDash \mathsf Q$. It remains to show that $\mathcal M$ satisfies the induction scheme for formulas of the form $t = s$.
    
    \begin{lemma}
        Let $t(x, y_1, \dots, y_n)$ be a $\mathcal L_{ar}$-term and $y_1, \dots, y_n \in M$ are fixed. We say that term $t(x, \vec y)$ is constant in $x$ if $\exists z \in M \forall x \in M (t(x, \vec y) = z)$. Then $t(x, \vec y)$ is constant in $x$ or $t(\omega, \vec y) = \omega$. In the latter case, $t(x, \vec y) \geqslant x$ for all $x \in M$.
    \end{lemma}
    
    \begin{proof}
        Trivial induction on terms from variables $x, y_1, \dots, y_n$.
    \end{proof}
    
    Using this lemma, one can easily prove the claim. Suppose, $\mathcal M \vDash t(0, \vec y) = s(0, \vec y)$ and $\mathcal M \vDash \forall x (t(x, \vec y) = s(x, \vec y) \rightarrow t(Sx, \vec y) = s(Sx, \vec y))$. Since $S\omega = \omega$, the latter means $\forall n \in \mathbb N \Big(\mathcal M \vDash t(n, \vec y) = s(n, \vec y) \rightarrow t(Sn, \vec y) = s(Sn, \vec y) \Big)$. By the usual induction we obtain $\forall n \in \mathbb N \Big(\mathcal M \vDash t(n, \vec y) = s(n, \vec y)\Big)$. If $t(x, \vec y)$ and $s(x, \vec y)$ are constant in $x$, then $\mathcal M \vDash \forall x (t(x, \vec y) = s(x, \vec y))$ and the induction axiom holds. If both $t$ and $s$ are not constant in $x$, then $t(\omega, \vec y) = \omega = s(\omega, \vec y)$, so, $\mathcal M \vDash \forall x (t(x, \vec y) = s(x, \vec y))$. Assume that $t$ is constant in $x$, $s$ is not constant in $x$. If $t(x, \vec y) = n \in \mathbb N$, then $t(n + 1, \vec y) = n \ne s(n + 1, \vec y) \geqslant n + 1$. So, $t(\omega, \vec y) = \omega = s(\omega, \vec y)$. 
    
    Finally, note that the constructed model falsifies $Sx \ne x$ and $x + z = z + y \rightarrow z = y$ (since $S\omega = \omega$ and $\omega + 0 = \omega + 1$).
    
    (ii) Note that $\mathsf I(\ne) \vdash Sx \ne x$ ($S0 \ne 0$ and $Sx \ne x \rightarrow SSx \ne Sx$ are consequences of $\mathsf{Q}$, then apply the induction for the formula $Sx \ne x$). 
\end{proof}

\begin{proposition}
    \begin{enumerate}
        \item[(i)] $\mathsf I(=) \nvdash \forall x \exists y (y^r \leqslant x \wedge \neg (Sy)^r \leqslant x)$ for all $r \geqslant 2$ (i.e. the existence of integer part of r-th roots is unprovable);
        \item[(ii)] $\mathsf I(=) \nvdash \mathsf I(\leqslant),  I(\nleqslant)$.
    \end{enumerate}
    
\end{proposition}

\begin{proof}
    (i) Consider the structure $\mathbb Z[X]^+ = \{a_n X^n + \dots + a_0 \in \mathbb Z[X]| a_n > 0 \vee a_n X^n + \dots + a_0 = 0\}$ with $S$, $+$ and $\cdot$ defined in the usual way and $f \leqslant g \leftrightharpoons f(x) \leqslant g(x)$ for all sufficiently large $x$. It is obvious that $\mathbb Z[X]^+ \vDash \mathsf Q$. 
    
    Let $t(x, \vec y), s(x, \vec y)$ be $\mathcal L_{ar}$-terms, $y_1, \dots, y_m \in \mathbb Z[X]^+$ are fixed. Suppose $t(0, \vec y) = s(0, \vec y)$ and $\forall x \big(t(x, \vec y) = s(x, \vec y) \rightarrow t(Sx, \vec y) = s(Sx, \vec y)\big)$. Then for all $k \in \mathbb N$ $t(k, \vec y) = s(k, \vec y)$. We can represent $t(x, \vec y) - s(x, \vec y)$ as $x^n P_n(\vec y) + \dots + P_0(\vec y)$, where $P_i(\vec y) \in \mathbb Z[\vec y]$. Considering $k = 0, 1, \dots, n$ we obtain
    $$
    \begin{pmatrix}
    1 & 0 & \cdots & 0 \\
    1 & 1 & \cdots & 1 \\
    \vdots & & \ddots & \vdots \\
    1 & n & \cdots & n^n
    \end{pmatrix}
    \begin{pmatrix}
    P_0(\vec y) \\
    \vdots \\
    P_n(\vec y)
    \end{pmatrix}
    =
    \begin{pmatrix}
    0 \\
    \vdots \\
    0
    \end{pmatrix}.
    $$
    
    Since the left matrix is invertible (it is a Vandermonde matrix), 
    $\begin{pmatrix}
    P_0(\vec y) \\
    \vdots \\
    P_n(\vec y)
    \end{pmatrix}
    =
    \begin{pmatrix}
    0 \\
    \vdots \\
    0
    \end{pmatrix}$. So, $\forall x \big(t(x, \vec y) = s(x, \vec y)\big)$ and $\mathbb Z[X]^+ \vDash \mathsf I (=)$.
    
    Let us now prove that $\mathbb Z[X]^+ \nvDash \exists y (y^r \leqslant X \wedge \neg (y + 1)^r \leqslant X)$ for $r \geqslant 2$. Consider any $y \in \mathbb Z[X]^+$. If $\deg y = 0$, then $(y + 1)^r \in \mathbb N$, so $(y + 1)^r < X$. If $\deg y \geqslant 1$, then $\deg y^r > 1$, so $y^r > X$.
    
    (ii) It is easy to see that $\mathbb Z[X]^+ \nvDash \mathsf I(\leqslant)$. Consider the induction axiom for the formula $x^r \leqslant y$. Suppose it holds in $\mathbb Z[X]^+$. Since $\mathbb Z[X]^+ \vDash 0^r \leqslant y, \neg \forall x (x^r \leqslant y)$, $\mathbb Z[X]^+ \vDash \exists x (x^r \leqslant y \wedge \neg (Sx)^r \leqslant y)$. So we obtain a contradiction. In the similar way we can prove $\mathbb Z[X]^+ \nvDash \mathsf I(\nleqslant)$.
\end{proof}

\begin{proof}[Proof of theorem 3.2.]
    (i) We claim that if some equation $s = t$ has a solution in a model of $\mathsf I(=)$, then it has a solution in the model $\mathcal M$ from Proposition 3.1. 
    
    Since in $\mathsf I(=)$ one can prove the commutativity, associativity and distributivity of addition and multiplication, all terms can be represented as 
    $$s(\vec x) = \sum\limits_{(i_1, \dots, i_n): i_1 + \dots + i_n \leqslant k} a_{i_1, \dots, i_n} x_1^{i_1} \dots x_n^{i_n},$$ 
    where $k$ is a natural number and all $a_{i_1, \dots, i_n}$ are numerals. It is clear that such a form can be found effectively. Let $\deg s := \max\{i_1 + \dots + i_n | a_{i_1, \dots, i_n} \ne 0\}$.
    
    Let us fix two terms $s(\vec x)$ and $t(\vec x)$. Consider three cases: 1) $\deg s = \deg t = 0$, 2) $\deg s > 0$, $\deg t = 0$ (or, symmetrically, $\deg s = 0$, $\deg t > 0$), 3) $\deg s > 0$, $\deg t > 0$.
    
    1) $s$ and $t$ are constants, so it is easy to check whether they are equal.
    
    2) Suppose there is $\mathcal N \vDash \mathsf I(=)$ such that $s(\vec x) = t(\vec x)$ for some $x_1, \dots, x_n \in N$. Let $s(\vec x) = \sum\limits_{\substack{(i_1, \dots, i_n): \\ i_1 + \dots + i_n \leqslant k}} a_{i_1, \dots, i_n} x_1^{i_1} \dots x_n^{i_n}$ in $\mathcal N$. Suppose that for some $j$ $x_j$ is a nonstandard. Then for all $i_1, \dots, i_n$ such that $i_1 + \dots + i_n \leqslant k$ either $i_j = 0$ or $a_{i_1, \dots, i_n} x_i^{i_1} \dots x_{j - 1}^{i_{j - 1}} x_{j + 1}^{i_{j + 1}} \dots x_{n}^{i_n} = 0$ (otherwise $a_{i_1, \dots, i_n} x_1^{i_1} \dots x_n^{i_n}$ and $s(\vec x)$ would be nonstandard, which is contradictory since $t(\vec x)$ is a standard). So, if we replace $x_i$ by $0$, $s(\vec x)$ will not change its value. Since that we can replace all of nonstandard $x_i's$ by $0$ and obtain a solution of the considered equation in $\mathbb N$ (and hence in $\mathcal M$). Also it is clear that all $x_i$ can be bounded by $t$.
    
    3) All such equations can be satisfied by taking $x_i = \omega$ ($s(\omega, \dots, \omega) = \omega = t(\omega, \dots, \omega)$).
    
    From this we can easily obtain an algorithm to decide whether $s = t$ is satisfiable in $\mathsf I(=)$.
    
    (ii) In fact, $Th_=(\mathbb N)$ can be deduced from $\mathsf Q$ and commutativity, associativity and distributivity of addition and multiplication. 
    
    Let us fix terms $s(\vec x)$ and $t(\vec x)$ such that $\mathbb N \vDash \forall \vec x (s(\vec x) = t(\vec x))$. As in (i), $s$ and $t$ can be represented as polynomials. Since they are equal in $\mathbb N$, they have equal coefficients and hence their equality is provable.
    
    (iii) We introduce the following model $\mathcal N$: $N = \mathbb N \cup \{\omega_0, \omega_1\}$, operations on natural numbers defined in the standard way, 
    \begin{itemize}
        \item $S \omega_i = \omega_i$, where $i \in \{0, 1\}$,
        \item $\omega_i + n = n + \omega_i = \omega_i$, where $n \in \mathbb N$, $i \in \{0, 1\}$,
        \item $\omega_i + \omega_j = \omega_{\max(i, j)}$, $i, j \in \{0, 1\}$,
        \item $0 \cdot \omega_i = \omega_i \cdot 0 = 0$, $i \in \{0, 1\}$,
        \item $n \cdot \omega_i = \omega_i \cdot n = \omega_i$, $i \in \{0, 1\}$, $n \in \mathbb N \setminus \{0\}$,
        \item $\omega_i \cdot \omega_j = \omega_{\max(i, j)}$, $i, j \in \{0, 1\}$.
    \end{itemize}
    
    $\mathcal N \vDash \mathsf Q$ and operations in $\mathcal N$ are commutative, associative and distributive, so $\mathcal N \vDash Th_=(\mathbb N)$. But $\mathcal N \vDash 0 + \omega_0 = \omega_0 \wedge \forall x (x + \omega_0 = \omega_0 \rightarrow S x + \omega_0 = \omega_0) \wedge \omega_1 + \omega_0 \ne \omega_0$, so $\mathcal N \nvDash \mathsf I (=)$.
    
\end{proof}

\begin{proposition}
    (i) $\mathbb Z[X]^+ \vDash \mathsf{I}(\ne)$;
    
    (ii) $\mathsf I(\ne) \nvdash \mathsf I(\leqslant), \mathsf I(\nleqslant)$.
\end{proposition}

\begin{proof}
    (i) We only need to prove $\mathbb Z[X]^+ \vDash Ind_{s(x, \vec y) \ne t(s, \vec y)}$, where $s$ and $t$ are terms. Fix these terms and $\vec y$. There are $P_n(\vec y), \dots, P_0(\vec y) \in \mathbb Z[X]$ such that $s(x, \vec y) - t(x, \vec y) = P_n(\vec y) x^n + \dots + P_0(\vec y)$. Suppose, $P_n(\vec y) 0^n + \dots + P_0(\vec y) = P_0(\vec y) \ne 0$, $\forall x (P_n(\vec y) x^n + \dots + P_0(\vec y) \ne 0 \rightarrow P_n(\vec y) (S x)^n + \dots + P_0(\vec y) \ne 0)$, but $\exists x \in \mathbb Z[X]^+: P_n(\vec y) x^n + \dots + P_0(\vec y) = 0$. Then, for all $k \in \mathbb N$ $P_n(\vec y) (x - k)^n + \dots + P_0(\vec y) = 0$ (since we can apply a contraposition to the step and the usual induction). 
    $$
    \begin{pmatrix}
    1 & x & \cdots & x^n \\
    1 & x - 1 & \cdots & (x - 1)^n \\
    \vdots & & \ddots & \vdots \\
    1 & x - n & \cdots & (x - n)^n
    \end{pmatrix}
    \begin{pmatrix}
    P_0(\vec y) \\
    \vdots \\
    P_n(\vec y)
    \end{pmatrix}
    =
    \begin{pmatrix}
    0 \\
    \vdots \\
    0
    \end{pmatrix}.
    $$
\end{proof}

Since this matrix is invertible (in $\mathbb Z (X)$), $P_n (\vec y) = \dots = P_0(\vec y) = 0$. So, $\forall x \in \mathbb{Z}[X] (s(x, \vec y) = t(x, \vec y))$, a contradiction.

(ii) We have already proved in Proposition 3.2 that $\mathbb Z^+ [X] \nvDash \mathsf I(\leqslant), \mathsf I(\nleqslant)$

\begin{proposition}
\begin{enumerate}
    \item[(i)] $\mathsf I (\leqslant) \nvdash x + y = y + x, x \cdot y = y \cdot x, S x \ne x$
    \item[(ii)] $\mathsf I (\leqslant) \nvdash \mathsf I (=), \mathsf I (\ne), \mathsf I (\nleqslant)$.
\end{enumerate}
\end{proposition}

\begin{proof}
    (i) Consider the model $\mathcal M$: $M = \mathbb N \cup \{\omega_0, \omega_1\}$ with operations defined as follows (on $\mathbb N$ all operations defined in the standard way):
    \begin{itemize}
        \item $S \omega_i = \omega_i$, where $i \in \{0, 1\}$
        \item $\omega_i + x = \omega_i$, $n + \omega_i = \omega_i$, where $i \in \{0, 1\}$, $x \in M$, $n \in \mathbb N$
        \item $0 \cdot x = x \cdot 0 = 0$, $\omega_i \cdot x = \omega_i$, $n \cdot \omega_i = \omega_i$, where $i \in \{0, 1\}$, $x \in M \setminus \{0\}$, $n \in \mathbb N \setminus \{0\}$
        \item $n \leqslant \omega_i$, $\omega_i \leqslant \omega_j$, where $i, j \in \{0, 1\}$, $n \in \mathbb N$
    \end{itemize}
    
    We prove that $\mathcal M$ is a model of $\mathsf{I(\leqslant)}$. It is not very hard to see that $\mathcal M \vDash \mathsf Q$. Let, for example, check $\mathcal M \vDash \forall x \forall y (x \cdot S y = x \cdot y + x)$. Fix $x, y \in M$. If $x, y \in \mathbb N$ or $x = 0$, it is obvious. Consider the case $x \in \mathbb N \setminus \{0\}$ and $y = \omega_i$: $x \cdot S y = x \cdot \omega_i = \omega_i = \omega_i + \omega_i = x \cdot y + y$. If $x = \omega_i$, then $x \cdot S y = \omega_i = x \cdot y + x$. 
    
    As in Proposition 3.1 we can formulate the analogous lemma about terms (i.e. for every term $t(x, \vec y)$ and fixed $y_1, \dots, y_n \in M$  $\exists z \in M \forall x \in M \: t(x, \vec y) = z$ or $t(\omega_i, \vec y) = \omega_i, i = 0, 1$ and $\forall x \in M \: t(x, \vec y) \geqslant x$) and end the proof of $\mathcal M \vDash \mathsf{I(\leqslant)}$ in a similar way.

    Now, $\mathcal M \vDash \omega_0 + \omega_1 \ne \omega_1 + \omega_0, \omega_0 \cdot \omega_1 \ne \omega_1 \cdot \omega_0, S \omega_0 = \omega_0$, as required.
    
    (ii) Easy follows from (i) since $\mathsf I(=) \vdash x + y = y + x, x \cdot y = y \cdot x$, $\mathsf I (\ne) \vdash S x \ne x$ and $\mathcal M \nvDash (\omega_0 \nleqslant 0) \wedge \forall x (\omega_0 \nleqslant x \rightarrow \omega_0 \nleqslant Sx) \rightarrow \forall x (\omega_0 \nleqslant x)$.
\end{proof}

\begin{proof}[Proof of the Theorem 3.1.]
    Follows from Propositions 3.1-3.4.
\end{proof}

\begin{proof}[Proof of the Theorem 3.3.]
    Firstly, we prove the following lemma.
    
    \begin{lemma}
        $\mathsf{I}(\ne) + \forall x \forall y (x + y = y + x)$ proves associativity, commutativity and distributivity of $+$ and $\cdot$. 
    \end{lemma}
    
    \begin{proof}
        \begin{itemize}
            \item Associativity of addition.
            
            Suppose there is $x, y, z$ such that $x + (y + z) \ne (x + y) + z$. Consider the formula $\varphi(x, y, z, t) := \Big((x + (y + z)) + ((x + y) + t) \ne ((x + y) + z) + (x + (y + t))\Big)$. Then, 
            $$\varphi(x, y, z, 0) \leftrightarrow (x + (y + z)) + (x + y) \ne ((x + y) + z) + (x + y) $$
            $\:\:\:\:\:\:\:\:\:\:\:\:\:\:\:\:\:\:\:\:\:\:\:\:\:\:\:\:\:\:\:\:\:\:\:\:\:\:\:\:\:\:\:\:\:\:\:\:\:\:\:\:\:\:\:\:\:\:\:\:\:\:\:\:\:
            \leftrightarrow x + (y + z) \ne (x + y) + z$
            
            (the latter equivalence is true since $\mathsf{I}(\ne) \vdash a + c = b + c \rightarrow a = b$). So, $\varphi(x, y, z, 0)$ is true.
            
            Suppose, $\varphi(x, y, z, t)$ is true, but $\varphi(x, y, z, St)$ is false. Then, 
            $$\neg\varphi(x, y, z, St) \leftrightarrow (x + (y + z)) + S((x + y) + t) = ((x + y) + z) + (x + S(y + t)) $$
            $\:\:\:\:\:\:\:\:\:\:\:\:\:\:\:\:\:\:\:\:\:\:\:\:\:\:\:\:\:\:\:\:\:\:\:\:\:\:\:\:\:\:\:\:\:\:\:\:\:\:\:\:\:\:
            \leftrightarrow S((x + (y + z)) + ((x + y) + t)) = S(((x + y) + z) + (x + (y + t)))$
            
            $\:\:\:\:\:\:\:\:\:\:\:\:\:\:\:\:\:\:\:\:\:\:\:\:\:\:\:\:\:\:\:\:\:\:\:\:\:\:\:\:\:\:\:\:\:\:\:\:\:\:\:\:\:\:
            \leftrightarrow (x + (y + z)) + ((x + y) + t) = ((x + y) + z) + (x + (y + t))$
            
            $\:\:\:\:\:\:\:\:\:\:\:\:\:\:\:\:\:\:\:\:\:\:\:\:\:\:\:\:\:\:\:\:\:\:\:\:\:\:\:\:\:\:\:\:\:\:\:\:\:\:\:\:\:\:
            \leftrightarrow \neg \varphi(x, y, z, t),$
            
            so, we have got a contradiction. Applying induction to the formula $\varphi$, we obtain $\forall t \:\varphi(x, y, z, t)$. Now, substitute $z$ instead of $t$: $$(x + (y + z)) + ((x + y) + z) \ne ((x + y) + z) + (x + (y + z)),$$
            
            contradiction with commutativity of addition.
            
            \item Right distributivity.
            
            Suppose there is $x, y, z$ such that $x(y + z) \ne xy + xz$. Consider the formula $\varphi(x, y, z, t) = \Big(x(y + z) + xy + xt \ne xy + xz + x(y + t)\Big)$ (since we have already proved associativity we can write terms as $s + t + r$).
            
            It is easy to see that $\varphi(x, y, z, 0)$ and $\neg\varphi(x, y, z, St) \rightarrow \neg \varphi(x, y, z, t)$ are true. By induction we obtain $\forall t \: \varphi(x, y, z, t)$. After substitution $t := z$ we obtain a contradiction with commutativity.
        \end{itemize}
        
        All other identities can be proven in the same way. Let's list only the formulas $\varphi(x, y, z, t)$.
        
        \begin{itemize}
            \item Left distributivity: $\varphi(x, y, z, t) = \Big((x + y)z + xz + yt \ne xy + yz + (x + y)t\Big)$;
            
            \item commutativity of multiplication: $\varphi(x, y, z, t) = \Big(xy + yt \ne yx + ty\Big)$;
            
            \item associativity of multiplication: $\varphi(x, y, z, t) = \Big(x(yz) + (xy)t \ne (xy)z + x(yt)\Big)$.
        \end{itemize}
    \end{proof}
    
    Let $\mathcal{M} \vDash \mathsf{I}(\ne) + \forall x \forall y (x + y = y + x)$. By lemma 3.2 $\mathcal{M}$ is a semiring that can be embedded in a ring $\widetilde{\mathcal{M}}$ (as in the proof of the Theorem 2.1). Let $f \in \widetilde{\mathcal{M}}[t]$. We prove by induction on $\deg f$ that the induction for the formula $f(x) = 0$ holds.
    
    If $\deg f = 0$, then $f(0) = f(x)$ for all $x \in M$. If $f(0) = 0$, then $\forall x (f(x) = 0)$.
    
    Let $\deg f = n > 0$, $f(0) = 0$ and $\forall x \in M (f(x) = 0 \rightarrow f(Sx) = 0)$. For $g \in \widetilde{\mathcal{M}}[t]$ we define $\tilde g (t) = g(St) - g(t)$. Denote by $\tilde g^{(k)}$ the $\tilde{(\cdot)}$ applied to $g$ $k$ times.
    
    \begin{proposition}
        $\forall k < n \forall x \in M (\tilde f^{(n - k)}(x) = 0)$.
    \end{proposition}
    
    \begin{proof}
        Note that $\tilde f^{(n)}$ is a constant (since $\deg \tilde g < \deg g$) and $\forall m, k \in \mathbb N (\tilde f^{(k)}(m) = 0)$ (since by usual induction $f(\mathbb N) = \{0\}$ and $g(x) = g(Sx) \rightarrow \tilde g(x) = 0$).
    
        Induction on $k$.
        
        If $k = 0$, then $\tilde f^{(n - k)}$ is a zero constant (by above observations).
        
        Let $k + 1 < n$. Then $\tilde f^{(n - (k + 1))}(0) = 0$ and $\tilde f^{(n - (k + 1))}(x) = 0 \rightarrow \tilde f^{(n - (k + 1))}(Sx) = 0$ (by the induction hypothesis). Since $\deg \tilde f^{(n - (k + 1))} < n$, we can apply induction axiom to the formula $\tilde f^{(n - (k + 1))}(x) = 0$ and obtain that for all $x \in M$ $\tilde f^{(n - (k + 1))}(x) = 0$.
    \end{proof}
    
    Now, suppose that there exists $x_0 \in M$ such that $f(x_0) \ne 0$. Consider the formula $f(x) - f(x_0) \ne 0$. Then, $f(0) - f(x_0) \ne 0$ and $f(x) - f(x_0) \ne 0 \rightarrow f(Sx) - f(x_0) \ne 0$ (since $(f(Sx) - f(x_0) - (f(x) - f(x_0)) = \tilde f^{(1)} (x) = 0$). Since $\mathcal{M} \vDash \mathsf{I}(\ne)$, we obtain $\forall x \in M (f(x) - f(x_0) \ne 0)$. It is a contradiction since we can substitute $x_0$ instead of $x$.
\end{proof}

\section{Remaining questions}

In this section we formulate some remaining problems.

\begin{problem}
    Does $\mathsf I(=)$ follow from $\mathsf I(\ne)$?
\end{problem}

If the answer to Problem 1 is negative (i.e. $\mathsf I(\ne) \nvdash \mathsf I(=)$), then by Theorem 3.3 any countermodel must have noncommutative addition. 

We introduce a structure with noncommutative operations. Informally speaking, this is an analogue of the $\mathbb Z[X]$, but with noncommutative operations. 
Clearly, since the commutativity of operations is provable in $\mathsf I(=)$, this structure will not be a model of $\mathsf I(=)$. 

Consider all formal sums of the form $a_1 X^{i_1} + \dots a_n X^{i_n}$ (the order of the sum is significant and we allow the sum to be empty), where $a_j \in \mathbb Z$ and $i_j \in \mathbb N$. We introduce the following reductions of such sums: 
$$a_1 X^{i_1} + \dots + a_{j - 1} X^{i_{j - 1}} + 0 X^{i_j} + a_{j + 1} X^{i_{j + 1}} + \dots + a_n X^{i_n} \mapsto a_1 X^{i_1} + \dots + a_{j - 1} X^{i_{j - 1}} + a_{j + 1} X^{i_{j + 1}} + \dots + a_n X^{i_n}$$

and 
$$a_1 X^{i_1} + \dots + a_{j - 1} X^{i_{j - 1}} + a_{j} X^{i_{j - 1}} + \dots + a_n X^{i_n} \mapsto a_1 X^{i_1} + \dots + (a_{j - 1} + a_{j}) X^{i_{j - 1}} + \dots + a_n X^{i_n}.$$

Let $\sim$ be the least equivalence relation, containing $\mapsto$. 

\begin{definition}
    A sum $A$ is in \textit{normal form} (NF) if there is no $B$ such that $A \mapsto B$.
\end{definition}

\begin{remark}
    It is easy to see that for every sum $A$ there is a unique sum $B$ such that $A \sim B$ and $B$ is in NF.
\end{remark}

We will consider these sums up to $\sim$.

Operations are introduced in the following way: 
$$(a_1 X^{i_1} + \dots + a_n X^{i_n}) + (b_1 X^{j_1} + \dots + b_m X^{j_m}) = a_1 X^{i_1} + \dots + a_n X^{i_n} + b_1 X^{j_1} + \dots + b_m X^{j_m},$$
$$S A = A + 1,$$

if $ b \geqslant 0$:
$$(a_1 X^{i_1} + \dots + a_n X^{i_n}) \cdot b X^j = \underbrace{(a_1 X^{i_1 + j} + \dots + a_n X^{i_n + j}) + \dots + (a_1 X^{i_1 + j} + \dots + a_n X^{i_n + j})}_{b \textit{ times}}, $$

if $b < 0$:
$$(a_1 X^{i_1} + \dots + a_n X^{i_n}) \cdot b X^j = \underbrace{(-a_n X^{i_n + j} - \dots - a_1 X^{i_1 + j}) + \dots + (-a_n X^{i_n + j} - \dots - a_1 X^{i_1 + j})}_{|b| \textit{ times}}, $$

$$A \cdot (b_1 X^{j_1} + \dots + b_m X^{j_m}) = A \cdot b_1 X^{j_1} + \dots + A \cdot b_m X^{j_m}.$$

As we can see, the result of the operations respects the equivalence relation introduced above.

Let us call this structure $\mathcal M$. To get from $\mathcal M$ a model of $\mathsf Q$, we need to take only the <<nonnegative>> (positive and zero) elements of $\mathcal M$. We call a sum \emph{positive} if in its normal form the sum of all coefficients before $X$'s with the greatest degree is positive. For example, $-X + X^2$ and $-X^2 + X + 2X^2$ are positive, but $X - X^2$ is not. It is easy to see that the sum and the product of any two nonnegative sums is nonnegative. As usual, we denote the substructure of nonnegative elements $\mathcal M^+$. 
Now, $\mathcal M^+ \vDash \mathsf Q$.

\begin{hypothesis}
    The introduced structure $\mathcal M^+$ is a model of $\mathsf I(\ne)$.
\end{hypothesis}

If this hypothesis turns out to be true, then $\mathsf I(\ne) \nvdash \mathsf I(=)$.

\begin{problem}
    Is $\mathsf I(=, \ne)$ equivalent to $\mathsf{IOpen}(=)$ (induction for quantifier-free formulas, containing only atomic formulas of the form $s = t$)?
\end{problem}

There is a following result on the alternative axiomatization of $\mathsf{IOpen(=)}$, which can help in solving this problem.

\begin{theorem}[\cite{shepherdson:1967}, Theorem 2]
    $\mathsf{IOpen(=)}$ is equivalent to the theory, consisting of $\mathsf Q$, commutativity, associativity and distributivity of addition and multiplication, and the scheme of axioms of the form
    $$\underline{d} x = \underline{d} x' \rightarrow \forall y \bigvee\limits_{i = 0}^{d - 1} ((y + i) x = (y + i) x')$$
    for all $d \geqslant 2$, where $\underline{d} = S^d(0)$.
\end{theorem}

\bibliographystyle{unsrt}
\bibliography{references}

\end{document}